\documentclass[12pt, a4paper]{article}
\usepackage{setspace}
\usepackage{geometry}
\usepackage{a4wide}
\usepackage[OT1]{fontenc}
\usepackage{amsthm,amsmath}
\usepackage[colorlinks,citecolor=blue,urlcolor=blue]{hyperref}
\usepackage{amssymb,amsmath,bm,mathrsfs,makeidx,amsfonts,graphicx,amsthm,multirow}
\usepackage{booktabs}

\newcommand\x{\mathbf x}
\newcommand\z{\mathbf z}
\newcommand\y{\mathbf y}
\newcommand\bR{{\mathbb R}}
\newcommand\bZ{{\mathbb Z}}

\newcommand\wh{\widehat }

\newcommand\cov{\mathrm{Cov}}
\newcommand\bN{{\mathbb N}}
\newcommand\pto{\stackrel {\mathbb P}\rightarrow}

\newcommand\e{{\mathbb E}}
\newcommand\p{{\mathbb P}}
\newcommand\var{{\rm Var}}
\newcommand\tr{{\mathrm{tr}}}

\newcommand\cB{{\mathcal B}}
\newcommand\bC{{\mathbb C}}

\renewcommand\Im{{\rm\,Im} }

\newtheorem{theorem}{Theorem}[section]%
\newtheorem{lemma}[theorem]{Lemma}%
\newtheorem{definition}[theorem]{Definition}%

\begin{document}

\begin{center}
\Large Marchenko-Pastur law for a random tensor model.   
\end{center}
\begin{center}
\large Pavel~Yaskov\footnote{Steklov Mathematical Institute of RAS, Moscow, Russia\\
 e-mail: yaskov@mi-ras.ru\\This work is supported by the Russian Science Foundation under grant 18-71-10097.}
 \end{center}

\begin{abstract}  We study the limiting spectral distribution of large-dimensional sample covariance matrices associated with symmetric random tensors formed by $\binom{n}{d}$ different products of $d$ variables chosen from $n$ independent standardized random variables. We find optimal sufficient conditions for this  distribution to be the Marchenko-Pastur law in the case $d=d(n)$ and $n\to\infty$. Our conditions reduce to $d^2=o(n)$  when the variables have uniformly bounded fourth moments. The proofs are based on a new concentration inequality for quadratic forms in symmetric random tensors and a law of large numbers for elementary symmetric random polynomials.
\end{abstract}

\begin{center}
\end{center}

\section{Introduction}

The paper studies the limiting behaviour of empirical spectral distributions of large-dimensional sample covariance matrices associated with the random tensor model investigated in \cite{BVH}. Namely, we consider sample covariance matrices of the form 
\begin{equation}
\label{e8}
\wh \Sigma_N=\frac{1}{N}\sum_{k=1}^N \x_{pk}\x_{pk}^\top,
\end{equation} 
where $\{\x_{pk}\}_{k=1}^N$ are i.i.d. copies of a random vector $\x_p$ in $\bR^p$ following the model below.

\begin{definition} \label{RT} \normalfont A random vector $\x_p$ in $\bR^p$ follows the random tensor model with parameters $d,n\in \bN, $ $d\leqslant n,$ if $p={{n}\choose {d}}$ and  the following holds.   
There exists a random vector $X=(X_\alpha)_{\alpha=1}^n$ in $\bR^n$ such that $\{X_\alpha\}_{\alpha=1}^n$ are independent, have zero mean and unit variance and $\x_p$ could be obtained by vectorizing the symmetric tensor $X^{\otimes d}$, i.e. the entries of $\x_p$ could be indexed by $d$-element subsets $i\subseteq\{1,\ldots,n\}$ and defined as products  $\prod_{\alpha \in i} X_\alpha.$ 
\end{definition}
Bryson, Vershynin, and Zhao \cite{BVH} show that the limiting spectral distribution of $\wh\Sigma_N$ is the Marchenko-Pastur (MP) law if $N\to\infty$, $\binom{n}{d}/N$ tends to a positive constant, the  fourth moments of $X_\alpha=X_{\alpha}(d,n)$ are uniformly bounded, and $d^3=o(n).$ However, they conjecture that the optimal condition is $d^2=o(n)$. We prove this conjecture in our paper. The proof is based on a general version of the MP theorem and a new concentration inequality for quadratic forms in symmetric random tensors. 

Concentration properties of quadratic forms provide a powerful tool to study the asymptotic behaviour of empirical spectral distributions of  sample covariance and related random matrices  (see \cite{A}, \cite{BZ}, \cite{DL}, \cite{GG}, \cite{GLPP}, \cite{EK} \cite{PP}, \cite{Y14}, \cite{Y18}). In particular, \cite{Y16} gives necessary and sufficient conditions for the MP theorem  in terms of concentration of certain quadratic forms. For various models of data, the concentration properties are established in \cite{BZ}, \cite{BVH}, \cite{L}, \cite{PP}, \cite{Y15}, \cite{Y18}, \cite{Y21}, among others. 

Lytova \cite{L}  obtains concentration inequalities for quadratic forms in non-symmetric random tensors $X^{(1)}\otimes \ldots \otimes X^{(d)}$ generated by $d$ independent copies of an isotropic random vector $X$ in $\bR^n$, quadratic forms in which concentrate around their means. In the simplest $L_2$ case, the results of \cite{L} are derived from the Efron-Stein inequality for the variance of a function in $\{X^{(k)}\}_{k=1}^d$.  When all $X=X(n)$, $n\in\bN,$ have independent entries with uniformly bounded fourth moments, Theorem 1.2 and Remark 4.1 in \cite{L} give a version of the MP theorem for the non-symmetric random tensor model under the assumption $d=o(n)$ (that seems to be optimal). Recently, Vershynin \cite{V} proves corresponding exponential concentration inequalities, assuming that $X$ has independent subgaussian entries.

In this paper, we derive $L_2$ concentration inequalities for quadratic forms in symmetric random tensors with optimal dependence on $d$ (at least, when $d=O(n)$). Our results improve those of \cite{BVH} and have some relevance to random chaoses. However, as noted in \cite{BVH}, known concentration inequalities for random chaoses \cite{A5}, \cite{A12}, \cite{GSS}, \cite{L6}, \cite{LL3}, \cite{L11} exhibit an unspecified (possibly exponential) dependence on the degree $d$, which is too bad for our problem. Also, comparing to the non-symmetric case, the strong dependence structure of the symmetric random tensor model highly complicates its analysis. In particular,
for the   quadratic form given by the squared $l_2$ norm, $\|X^{(1)}\otimes \ldots \otimes X^{(d)}\|^2$ is the product $\prod_{k=1}^d\|X^{(k)}\|^2$ of independent random variables and $\|X^{\otimes d}\|^2$ is the elementary symmetric polynomial of order $d$ in the squared entries of $X$. The latter plays a key role in the context of the MP theorem for the symmetric random tensor model, as the necessary condition for  the theorem is  given by $\binom{n}{d}^{-1}\|X^{\otimes d}\|^2\to 1$ in probability when $X=X(n)$, $d=d(n)$, and $n\to\infty$ (see Section \ref{results}). This condition could be viewed as a law of large numbers for elementary symmetric polynomials in independent nonnegative random variables. The asymptotic behaviour of such polynomials is thoroughly studied in \cite{EH}, \cite{HS}, \cite{M}, \cite{S}. Using the saddle-point approximation method of \cite{HS}, we find necessary and sufficient conditions (in terms of $d$ and $n$) for the above law of large numbers in the i.i.d. case.

The paper is structured as follows. Section \ref{results} contains main results. The proofs are deferred to Section \ref{proofs}. Auxiliary results are proved in the Supplementary Material.

\section{Main results}\label{results}
Let us introduce some notation.  For all $p\geqslant 1$, let $\x_p$ be a random vector in $\bR^p$ and let  $\bR^{p\times p}$ be  the set of all real $p\times p$ matrices. 
For $A\in \bR^{p\times p}$, denote its spectral norm by $\|A\|$ and, for symmetric $A$, let $\mu_A$ be its empirical spectral distribution defined by
$\mu_{A}:= p^{-1}\sum_{i=1}^p \delta_{\lambda_i},$
where $\{\lambda_i\}_{i=1}^p$ is the set of eigenvalues of $A$ (here we allow $\lambda_i=\lambda_j$ for $i\neq j$) and $\delta_\lambda$ stands for a Dirac measure with mass at $\lambda\in\bR$.  Also, put $\bC_+:=\{z\in\bC:\Im(z)>0\}$. Denote further by $\cB(\bR_+)$ the Borel $\sigma$-algebra of $\bR_+$, by $|S|$ the cardinality of a set $S$,  by $\mathbf{1}(S) $ the indicator function of $S$, and by $[n]_d$ the set $\{i\subseteq\{1,\ldots,n\}:|i|=d\}$.  For a set of random variables $X_n$, write $X_n=o_\p(1)$ if  $X_n $ tends to 0 in probability as $n\to\infty$, and $X_n=O_\p(1)$ if the set of variables $X_n$ is stochastically bounded. All random elements below  will be defined on the same probability space.

Recall the definition of the Marchenko-Pastur law $\mu_\rho$ ($\rho>0$), i.e.
\[d\mu_\rho=\max\{1-1/\rho,0\}\,d\delta_0+\frac{\sqrt{(a_+-x)(x-a_-)}}{2\pi x\rho   }{\mathbf{1}}\big( a_-\leqslant x\leqslant a_+\big)\,dx\,\text{ for }\,a_{\pm}=(1\pm\sqrt{\rho})^2,  \]
and a version of the Marchenko-Pastur theorem (going back to \cite{MP}) under the following general assumption\footnote{For independent but not identically distributed   $\{\x_{pk}\}_{k=1}^N$ given in the definition of $\wh\Sigma_N$, one can replace (A) by its averaged version as in Section 2 of \cite{Y14}, when stating the Marchenko-Pastur theorem. However, to simplify the presentation, we consider only the i.i.d. case in this paper.} that quadratic forms in $\x_p$ weakly concentrate around some values:

(A) $(\x_p^\top A_p \x_p -\tr(A_p ))/p\pto 0$ as $p\to\infty$ for all sequences of  symmetric positive semidefinite $A_p\in\bR^{p\times p}$ with $\|A_p\|\leqslant 1$.

\begin{theorem}\label{th:1}
Let $p=p(N)\in\bN$ satisfy  $p/N\to \rho>0$ as $N\to\infty$. Let also $\x_p$ be a random vector in $\bR^p$ for every $p$. If $\rm (A)$ holds, then
\begin{equation}\label{eq:mpp}
\p(\mu_{\wh \Sigma_N}\to  \mu_\rho\text{ weakly, } N\to\infty)=1.
\end{equation}
Furthermore, if \eqref{eq:mpp} holds and $\e\x_p\x_p^\top=I_p$  for all $p$, then $\x_p^\top\x_p/p\pto 1$ as $p=p(N)\to\infty$.
\end{theorem}
The first part of the theorem follows from Theorem 2 in \cite{Y18} (where we take $\Sigma_T=I_T$) and the classical Marchenko-Pastur theorem for $\x_p$ having i.i.d. Gaussian entries. The second part of the theorem follows from Theorem 1.1 in \cite{Y16}.

To apply Theorem \ref{th:1},  one has to check (A). For the random tensor model, this could be done via the following concentration inequality, which is our first main result (for its proof, see Section \ref{proofs}).

\begin{theorem}\label{th:2} Let $d,n\in\bN$, $n\geqslant 16d,$ and $p=\binom{n}{d}$. If $\x_p$ is a random vector in $\bR^p$ following the random tensor model and  $A_p\in\bR^{p\times p}$, then
\begin{equation*}
 \var(\x_p^\top A_p\x_p)\leqslant p\,\tr(A_pA_p^\top) \begin{cases} (1+K_nd/n)^d-1&\text{if }A_p\text{ is diagonal},\\
(1+2K_nd/n)^d(16d/n) (d\wedge 8)&\text{if }A_p\text{ is zero-diagonal},\\
 64K_n d^2/n&\text{if }A_p\text{ is arbitrary, }2K_nd^2 \leqslant n, \end{cases}
\end{equation*}
where $K_n=\max\{\e X_\alpha^4:1\leqslant \alpha\leqslant n\}$ with  $ X_\alpha$ constituting $\x_p$ according to Definition \ref{RT}.
\end{theorem}
Theorem \ref{th:2} improves the corresponding upper bound in Theorem 1.9 in \cite{BVH} by a factor of $n^{-1/3}$ (up to some constants). The latter follows from
 \[\frac{K_nd^2}{ n}=\frac{ (K_n^{1/2}d /n^{1/3})^2}{n^{1/3}}\quad\text{and}\quad\tr(A_pA_p^\top)\leqslant \|A_pA_p^\top\| \tr(I_p)=p\|A_p\|^2 .\] Note also that if $X_\alpha$ are i.i.d. over $\alpha$, then  $\var(\x_p^\top\x_p)\geqslant p^2 (K_n-1)d^2/n$  by Theorem 5.2 of Hoeffding \cite{H}. This shows the sharpness of our bound.

Theorem \ref{th:2}  guarantees that (A) holds for any sequence of random vectors  $\x_p $ following the random tensor model with  $p=\binom {n}{d}$ and $X_\alpha=X_\alpha(d,n)$, when  $K_nd^2=o(n)$ and $n\to\infty$ (here $d, n $ may depend on some parameter $N$ that goes to infinity). This along with Theorem \ref{th:1} gives a version of the Marchenko-Pastur theorem for the random tensor model under the fourth moment condition, extending Theorem 1.5 of \cite{BVH}, where it is assumed that $d^3=o(n)$ and $K_n=O(1)$. We can state a more general result, assuming only that the second moments are finite.

\begin{theorem} \label{th:3}
Let $d=d(N),n=n(N)\in\bN$, $N\in\bN$, satisfy $d\leqslant n$ and $\binom{n}{d}/N\to \rho>0$, whereinafter all limits are with respect to $N\to\infty$. Assume also that, for each $p=\binom{n}{d}$, $\x_p$ is a random vector in $\bR^p$ that follows the random tensor model with parameters $d,n$ and $X_\alpha=X_\alpha(d,n),$ $\alpha=1,\ldots,n$.
Then \eqref{eq:mpp} follows from \eqref{eq:14}, where 
\begin{equation}\label{eq:14}
d \e X_\alpha^2\mathbf{1}(dX_\alpha^2>n)\to0\text{ and }\dfrac{d^2}n \e X_\alpha^4 \mathbf{1}(dX_\alpha^2\leqslant n)\to0\text{ uniformly in $\alpha\in\{1,\ldots,n\}$.}
\end{equation}
Conversely, \eqref{eq:mpp} implies \eqref{eq:14} if, for all $d,n$, $\{X_\alpha\}_{\alpha=1}^n$ are independent copies of a random variable $X$ not depending on $d,n$ and such that $\e X=0,$ $\e X^2=1>\p(X^2=1)$.
 \end{theorem}
Theorem \ref{th:3} is proved in Section \ref{proofs}. The sufficiency part of the theorem follows from Theorem \ref{th:1} and \ref{th:2}. The necessity part follows from Theorem \ref{th:1} and the following law  of large numbers (LLN) for elementary symmetric polynomials of the form
 \[S_n^{(d)}:=\sum_{1\leqslant i_1<\ldots <i_d\leqslant n}Z_{i_1}\cdots Z_{i_d}, \quad d,n\in\bN.\]
 \begin{theorem} \label{th:4}
Let $Z,Z_1,Z_2,\ldots$ be i.i.d. nonnegative nondegenerate random variables. If $\e Z=1$ and $d=d(n)\in \{1,\ldots, n\}$ for  $n\in\bN,$ then the following are equivalent as $n\to\infty$:

{\rm (i)} $S_n^{(d)}/\binom{n}{d}\pto 1 $, 

{\rm (ii)} $d(S_n/n-1)\pto 0$ for $S_n=S_n^{(1)}$,

\vspace*{0.4em}{\rm (iii)} $d \e Z \mathbf{1}(dZ>n)\to 0 $  and $ d^2\e Z^2 \mathbf{1}(dZ\leqslant n)=o(n)$.

\end{theorem}

Theorem \ref{th:4} is proved in Section \ref{proofs}. It gives necessary conditions for LLN for the $U$-statistic $U_{n}^{(d)}:=S_n^{(d)}/\binom{n}{d}$. The asymptotic distribution  of this $U$-statistic  is thoroughly studied in \cite{EH}, \cite{HS}, \cite{M}, \cite{S} for various asymptotic regimes. In particular, as is shown in \cite{EH}, under linear norming, the distribution may differ significantly for the cases $d^2=o(n)$, $d^2\sim c n$, and $d^2/n\to \infty$. However, to the best of our knowledge, there is no result in the literature, which gives (i)$\Rightarrow$(ii) under no assumptions on $d,n$, and higher-order moments of $Z$.\footnote{If $U_n^{(d)}\to 1$ in $L_2$, one can prove (ii) by using the following fact from the theory of $U$-statistics: $d(S_n/n-1)$ is the Hájek projection of $U_n^{(d)}-1$ in $L_2$, i.e. the projection on the set of sums $\sum_{k=1}^nf_k(Z_k)$, in particular, $\e|d(S_n/n-1)|^2\leqslant \e|U_n^{(d)}-1|^2$. However, this argument  works well only for $L_2$ convergence.} To prove (i)$\Rightarrow$(ii), we first show that (i) implies  $d^2=o(n)$ and then use the asymptotic representation for $S_n^{(d)}$ from Lemma \ref{l4} below. In the case of positive  $Z$, the lemma follows from formula (8) in \cite{S}, but the latter is stated without the proof. For completeness, we prove Lemma \ref{l4} in the Supplementary Material.
\begin{lemma}\label{l4} Under the conditions of Theorem \ref{th:3}, let $d^2=o(n)$ and $d\to\infty$. Then \begin{equation}\label{eq:17}
\ln\frac{S_n^{(d)}}{\binom{n}{d}}= \sum_{k=1}^n \ln(Z_k/\rho+1 )-d+d \ln(\rho d/n)+o_\p(1),\quad n\to\infty, 
\end{equation}
where $\rho=\rho(n,Z_1,\ldots,Z_n)$ is the unique solution of the equation
$ \sum_{k=1}^n  \rho/( Z_k+\rho)=n-d $
    if  such solution exists and $\rho=1$ otherwise.
\end{lemma}
\section{Proofs}
\label{proofs}

\begin{proof}[Proof of Theorem \ref{th:1}.] For brevity, we will write $\x$, $A$, $K$ instead of $\x_p,$ $A_p$, $K_n$. Let further $\x=(x_i)$, $A=(a_{ij})$, where $i,j$ run all elements of $[n]_d$($=\{i\subseteq\{1,\ldots,n\}:|i|=d\}$). Assume w.l.o.g. that $A$ is non-zero (the result is trivial otherwise). First, consider the case of diagonal $A$. We have
$\var(\x^\top A\x)=\sum a_{ii}a_{jj}\cov(x_i^2,x_j^2),$
where the sum is over all $i,j\in[n]_d$. Note that $\cov(x_i^2,x_j^2)=0$ if $i\cap j=\varnothing$ and
\[|\cov(x_i^2,x_j^2)|\leqslant  \e x_i^2 x_j^2 =\prod_{\alpha\in i\Delta j}\e X_\alpha^2\prod_{\beta\in i\cap j}\e X_\beta^4\leqslant K^{|i\cap j|}\text{ for any $i,j.$}\]
 The latter along with the Cauchy inequality yields 
\[\var(\x^\top A\x)\leqslant \frac12\sum_{i,j:i\cap j\not=\varnothing }(a_{ii}^2+a_{jj}^2)K^{|i\cap j|}= \sum_{i,j:i\cap j\not=\varnothing } a_{ii}^2 K^{|i\cap j|}=\sum_{t=1}^d K^t\sum_{i } a_{ii}^2 \sum_{j: |i\cap j|=t }1.\]
If $|i\cap j|=t(\leqslant d)$ is fixed, then there are ${{d}\choose{t}}$ choices for choosing $i\cap j$ for any given $i$ and $\binom{n-d}{d-t}$ choices for choosing $j\setminus i$ for any given $i$ and $i\cap j$. Therefore, the very last sum is equal to $\binom{d}{t}\binom{n-d}{d-t}$.
As $d\leqslant n,$ we see that \[\frac{\binom{n-d}{d-t}}{\binom{n}{d}}=\frac{ d!/ (d-t)!}{n!/(n-t)!}\frac{(n-d)!/(n-d-(d-t))!}{(n-t)!/(n-d)! }\leqslant \prod_{s=0}^{t-1}\frac{ d-s }{n-s}\prod_{r=0}^{d-t-1}  \frac{n-d-r}{n-t-r}  \leqslant \bigg(\frac dn\bigg)^{t} 1^{d-t}.\]
Combining the above bounds and recalling that $p=\binom{n}{d}$, we derive that
\[\var(\x^\top A\x)\leqslant p\, \tr(AA^\top)\sum_{t=1}^d \binom{d}{t}(Kd/n)^t= p\,\tr(AA^\top)((1+Kd/n)^d-1).\]

Consider the case of  zero-diagonal $A$.
Let  $\Delta_t:= \sum a_{ij}x_{i}x_{j}\mathbf{1}(|i\cap j|=t)$,   $t\in[0,d)\cap\bZ_+$, where the sum is over  $i,j\in [n]_d$.
By the triangle inequality for the norm $\|\cdot\|_2=\sqrt{\e|\cdot|^2}$,
\[\sqrt{\var(\x^\top A\x)}=\sqrt{\e|\x^\top A  \x|^2}= \Big\|\sum_{t=0}^{d-1}\Delta_t\Big\|_2\leqslant  \sum_{t=0}^{d-1} \|\Delta_t\|_2= \sum_{t=0}^{d-1} \sqrt{\e|\Delta_t|^2}. \]
Let us estimate $\e|\Delta_t|^2$ for any fixed $t$. By definition,
\[\e|\Delta_t|^2=\sum_{\substack{i,j,k,l:|i\cap j|=|k\cap l|=t }} a_{ij}a_{kl}\e  x_{i}x_{j}x_kx_l.\]
The product $x_ix_jx_kx_l$ is equal to
\[\prod_{\alpha\in \Lambda_1}X_\alpha\prod_{\beta\in \Lambda_2}X_\beta^2\prod_{\gamma\in \Lambda_3}X_\gamma^3\prod_{\delta\in \Lambda_4}X_\delta^4,\]
where $\Lambda_c=\Lambda_c(i,j,k,l),$ $ c=1,\ldots,4,$ contains all $\alpha\in i\cup j\cup k\cup l$ that are covered by exactly $c$ sets among $i,j,k,l$. If $|\Lambda_1|\neq 0$, then $\e x_ix_jx_kx_l=0$. When  $|\Lambda_1|=0$, it follows from the independence of $\{X_\alpha\}_{\alpha=1}^n$ and the Cauchy–Schwarz inequality that 
\[|\e x_ix_jx_kx_l|=\prod_{\gamma\in \Lambda_3}|\e X_\gamma^3|\prod_{\delta\in \Lambda_4}\e X_\delta^4\leqslant 
 K^{|\Lambda_4|}\prod_{\gamma\in \Lambda_3}\sqrt{\e X_\gamma^4\e  X_\gamma^2}\leqslant  K^{s},\]
 where $s:=|\Lambda_3|/2+|\Lambda_4|\in\bZ_+$ if $|\Lambda_1|=0$ by $|\Lambda_1|+2|\Lambda_2|+3|\Lambda_3|+4|\Lambda_4|=4d$. We get that
\[\e|\Delta_t|^2\leqslant\sum_{s\in\bZ_+}K^s\sum_{(i,j,k,l)\in \Gamma(s,t)} |a_{ij}a_{kl}| \]
for $\Gamma(s,t):=\{ (i,j,k,l): |i\cap j|=|k\cap l|=t,|\Lambda_1|=0,|\Lambda_3|/2+|\Lambda_4|=s\}$ and $\Lambda_c=\Lambda_c(i,j,k,l)$, $c=1,3,4$.
 
By symmetry, for all given $i,j$ with $|i\cap j|=t$, the sets $\{(k,l): (i,j,k,l)\in \Gamma(s,t)\}$ have the same cardinality. Let us denote this cardinality by $\gamma(s,t)$. By the Cauchy inequality,
\[\sum_{(i,j,k,l)\in \Gamma(s,t)} |a_{ij}a_{kl}|\leqslant\frac12\sum_{(i,j,k,l)\in \Gamma(s,t)} (a_{ij}^2 + a_{kl}^2)= \sum_{(i,j,k,l)\in \Gamma(s,t)} a_{ij}^2=  \gamma(s,t) \sum_{i,j:|i\cap j|=t} a_{ij}^2.\]
Combining the above bounds (for all $t$) along with the Cauchy-Schwartz inequality gives 
\[\sqrt{\var(\x^\top A\x)}\leqslant   \sum_{t=0}^{d-1}\sqrt{\sum_{i,j:|i\cap j|=t} a_{ij}^2}\sqrt{ \sum_{s\in\bZ_+} K^s \gamma(s,t)}\leqslant \sqrt{\tr(AA^\top)}\sqrt{\sum_{t=0}^{d-1}\sum_{s\in\bZ_+} K^s \gamma(s,t)}. \]
\begin{lemma}\label{l3}  Under the above notations, for all $t\in[0,d)\cap \bZ_+ $ and $s\in\bZ_+,$ we have 
\[\gamma(s,t)\leqslant   2^{4(d-t)} \binom{n}{d} \binom{t}{s}2^s(d/n)^{d-(t-s)}{\bf 1}(s\leqslant t). \]
\end{lemma}
Lemma \ref{l3} is proved in the Supplementary Material.
For $p=\binom{n}{d}$, it implies that   
\[\frac{\var(\x^\top A\x)}{ p\,\tr(AA^\top)}\leqslant  \sum_{t=0}^{d-1 } \bigg(\frac{2^4d}n\bigg)^{d-t }\sum_{s=0}^{t} \binom{t}{s}\bigg(\frac{2Kd}n\bigg)^s =
  \sum_{t=0}^{d-1 } \bigg(\frac{16 d}n\bigg)^{d-t }\bigg(1+\frac{2Kd}n\bigg)^t.\]
If $S$ is the last sum, then $S\leqslant (16 d^2/n)(1+ 2Kd/n)^d$. Also, as $n\geqslant 16 d$ and $K\geqslant (\e X_1^2)^2= 1$,
\[S=  \frac{16d}{n} \frac{(1+2Kd/n)^d-(16d/n)^d} {1+2Kd/n- 16d/n}\leqslant \frac{16d}{n} \frac{(1+2Kd/n)^d} {1-14d/n}\leqslant   \frac{2^7d}{n}  (1+2Kd/n)^d.\] This proves the desired bound for zero-diagonal $A$. 
The variance bound for an arbitrary matrix $A$ follows from the corresponding bounds for $A_0=(a_{ij}\mathbf{1}(i=j))_{i,j\in[n]_d}$ and $A_1=A-A_0$ along with the inequalities  $\sqrt{\var(\x^\top A\x)}\leqslant \sqrt{\var(\x^\top A_0\x)}+\sqrt{\var(\x^\top A_1\x)}$, $K\geqslant 1,$ $(1+x)^d\leqslant e^{dx}$ and $(1+x)^d-1\leqslant dxe^{dx}$, $x\geqslant 0$.

\end{proof}
\begin{proof}[Proof of Theorem \ref{th:3}] Suppose (ii) holds. To apply Theorem \ref{th:1}, we will verify (A). For each $p=\binom{n}{d}$, let $A_p\in\bR^{p\times p}$ be a positive semidefinite  symmetric matrix with $\|A_p\|\leqslant 1$. Consider a random vector $\y_p$ defined as $\x_p$ in Definition \ref{RT} with $X_\alpha$ replaced by $Y_\alpha:=X_\alpha \mathbf{1}(dX_\alpha^2\leqslant n)$. As $\{\x_{p}^\top A_p \x_{p}\neq \y_{p}^\top A_p \y_{p}\}\subseteq \bigcup_{\alpha=1}^n\{dX_\alpha^2>n\}$, \eqref{eq:14} yields that
\[\p(\x_{p}^\top A_p \x_{p}\neq \y_{p}^\top A_p \y_{p})\leqslant  \sum_{\alpha=1}^n\p(dX_\alpha^2>n)\leqslant  n\max_{1\leqslant \alpha\leqslant n} \e (dX_{\alpha}^2/n)\mathbf{1}(dX_\alpha^2>n)=:L_{n,d}\to0.\]
Thus, it sufficient to check (A) for $ \x_{p}$ replaced by $\y_{p}$. We will do it by showing that for $\z_p$ defined as $\x_p$ in Definition \ref{RT} with $X_\alpha$ replaced by $Z_\alpha:=Y_\alpha -\e Y_\alpha$, the following holds:

(a) $(\z_{p}^\top A_p \z_{p}-\y_{p}^\top A_p \y_p)/p\pto 0$, 

(b) $(\z_{p}^\top A_p \z_{p}-\e  \z_p^\top A_p \z_{p})/p\pto0$,

(c) $(\e  \z_p^\top A_p \z_{p} -\tr(A_p))/p\pto0$.

Let us prove (a). Set $\z_{p,0}:=\y_p$ and define $\z_{p,k }$, $k=1,\ldots,n,$ as $\y_p$ above with $Y_\alpha$ replaced by $Z_\alpha$ for $\alpha=1,\ldots , k$. We have
\[\e |\z_{p}^\top A_p \z_{p}-\y_{p}^\top A_p \y_p|\leqslant \e |\z_{p,n}^\top A_p \z_{p,n}-\z_{p,0}^\top A_p \z_{p,0}|\leqslant  \sum_{k=1}^n\e\Delta_{p,k},\]
where $\Delta_{p,k}:= |\z_{p,k}^\top A_p\z_{p,k} -\z_{p,k-1}^\top A_p \z_{p,k-1} |=|\|A_p^{1/2}\z_{p,k}\|^2-\|A_p^{1/2}\z_{p,k-1}\|^2|$. To estimate $\e\Delta_{p,k}$ for any given $k$, we will use the inequality
\[\e|\xi_0^2-\xi_1^2|\leqslant  \sqrt{ \e(\xi_0 -\xi_1)^2}\sqrt{\e(\xi_0+\xi_1)^2}\leqslant 2 \max_{q=0,1} \sqrt{\e\xi_q^2}\sqrt{ \e(\xi_0 -\xi_1)^2}\]
valid for any $\xi_0,$ $\xi_1$ in $L_2$. Taking $\xi_q=\|A_p^{1/2}\z_{p,k-q}\|$ and using that $\|A_p^{1/2}\|=\sqrt{ \|A_p \|}\leqslant 1$, we get
$|\xi_0-\xi_1|\leqslant\|A_p^{1/2}(\z_{p,k}-\z_{p,k-1})\|\leqslant\|\z_{p,k}-\z_{p,k-1}\|$  and
 $\e \xi_q^2\leqslant\e\|\z_{p,k-q}\|^2 \leqslant p$ as 
\begin{equation}\label{eq:11}
\e Z_\alpha^2=\var (Y_\alpha)\leqslant \e Y_\alpha^2=\e X_\alpha^2 \mathbf{1}(dX_\alpha^2\leqslant n)\leqslant\e X_\alpha^2=1.
\end{equation}
Thus, $\e \Delta_{p,k}\leqslant   2 \sqrt{p} \sqrt{ \e \|\z_{p,k}  - \z_{p,k-1} \|^2}$. The entries $\z_{p,k}  - \z_{p,k-1}$ have the form 
\[\mathbf{1}(k\in i)\prod_{\alpha:\,\alpha \in i ,\,1\leqslant \alpha<k }Z_\alpha \cdot \e Y_k \cdot \prod_{\alpha:\,\alpha \in i,\,k<\alpha\leqslant n}Y_\alpha \]
for $i\in[n]_d$($=\{j\subseteq \{1,\ldots,n\}:  |i|=d\}$), here $\prod_{i\in\varnothing}=1$. By \eqref{eq:11},
\[\e \|\z_{p,k}  - \z_{p,k-1} \|^2\leqslant  (\e Y_k)^2 \sum_{i\in[n]_d: k \in i  }\mathbf{1}(k\in i)=(\e Y_k)^2 \binom {n-1}{d-1}=\]\[=(\e Y_k)^2 \binom {n}{d}\frac{d}{n}=\frac{pd}{n}(\e X_k \mathbf{1}(dX_k^2\leqslant n))^2.\]
It follows from $\e X_k=0$, that $\e X_k \mathbf{1}(dX_k^2\leqslant n)=-\e X_k \mathbf{1}(dX_k^2>  n)$ and
\begin{equation}
\label{eq:12}
|\e X_k \mathbf{1}(dX_k^2\leqslant n)|\leqslant  \e |X_k| \mathbf{1}(dX_k^2>  n) \leqslant \sqrt{d/n}\,  \e X_k^2\mathbf{1}(dX_k^2>  n)  
\end{equation}
Combining the above estimates yields
\[\e |\z_{p}^\top A_p \z_{p}-\y_{p}^\top A_p \y_p| \leqslant  \sum_{k=1}^n\e\Delta_{p,k}\leqslant   \frac{2pd}{n}\sum_{k=1}^n\e X_k^2\mathbf{1}(dX_k^2>n)\leqslant 
2p L_{n,d}=o(p).\]
This implies (a). 

Let us prove (b). Applying Theorem \ref{th:2} to $\bar \z_p$ defined as $\x_p$ in Definition \ref{RT} with $X_\alpha$ replaced by $Z_\alpha/\sqrt{\e Z_\alpha^2}$, we get
\[\var(\z_{p}^\top A_p \z_{p})= \var(\bar\z_{p}^\top D_p^{1/2}  A_p  D_p^{1/2} \bar \z_{p})\leqslant \frac{ 64 K_n pd^2}n \tr((D_p^{1/2}  A_p  D_p^{1/2})^2)\text{ if }2K_nd^2\leqslant n,\]
where  $D_p:=\e\z_p\z_p^\top\in\bR^{p\times p}$ is  diagonal with diagonal entries $\prod_{\alpha \in i}\e Z_\alpha^2$, $i\in[n]_d$, and 
\[K_n:=\max_{1\leqslant \alpha \leqslant n}\frac{\e Z_\alpha^4}{(\e Z_\alpha^2)^2}\leqslant \frac{16\max\limits_{1\leqslant \alpha \leqslant n} \e X_\alpha^4\mathbf{1}(dX_\alpha^2\leqslant  n)}{\min\limits_{1\leqslant \alpha \leqslant n} (\e Z_\alpha^2) ^2},\]
here we have used that $\e Z_\alpha^4=\e (Y_\alpha-\e Y_\alpha)^4\leqslant 8(\e Y_\alpha^4+(\e Y_\alpha)^4)\leqslant 16\e Y_\alpha^4$.
By \eqref{eq:12} and $L_{n,d}=o(1)$, we have  uniformly in $\alpha$,
\[\e Y_\alpha = \e X_\alpha \mathbf{1}(dX_\alpha^2\leqslant n)=o((dn)^{-1/2}),\]
\begin{equation}
\label{eq:13}
\e Z_\alpha^2 = \e Y_\alpha^2-(\e Y_\alpha)^2= 1-\e  X_\alpha^2 \mathbf{1}(dX_\alpha^2>n)+o((dn)^{-1/2})=1+o(d^{-1})
\end{equation}
As a result, by \eqref{eq:14}, $d^2K_n/n =o(1)/(1+o(d^{-1}))=o(1)$. In addition, 
\[\tr((D_p^{1/2}  A_p  D_p^{1/2})^2)\leqslant p\|(D_p^{1/2}  A_p  D_p^{1/2})^2\|\leqslant p\|A_p\|^2 \|D_p^{1/2}\|^4\leqslant p\|D_p^{1/2}\|^4=\]\[= p\max\limits_{i\in[n]_d} \Big( \prod_{\alpha \in i}\e Z_\alpha^2 \Big)^2\leqslant p \Big(\max\limits_{1\leqslant \alpha\leqslant  n}\e Z_\alpha^2\Big)^{2d}= p (1+o(d^{-1}))^{4d}=p(1+o(1)).\]
Finally, we conclude that $K_nd^2=o(n)$ and
$\var(\z_{p}^\top A_p \z_{p})\leqslant p^2 o(1)$, which implies (b). 

The relation (c) follows from the fact that 
\[\e  \z_p^\top A_p \z_{p}= \tr(A_p D_p)=\tr(A_p)+\tr(A_p (D_p-I_p))\]
and $\tr(A_p (D_p-I_p))=o(p)$. Here the last equality could be derived from Von Neumann's trace inequality and \eqref{eq:13} as follows,
\[p^{-1}|\tr(A_p (D_p-I_p))|\leqslant \|A_p\| \|D_p-I_p\|\leqslant \|D_p-I_p\|=\max\limits_{i\in[n]_d} \Big| \prod_{\alpha \in i}\e Z_\alpha^2 -1 \Big|\leqslant\]\[\leqslant      |\max\limits_{1\leqslant \alpha\leqslant  n}(\e Z_\alpha^2)^d-1|+  
|\min\limits_{1\leqslant \alpha\leqslant  n}(\e Z_\alpha^2)^d-1|=   |(1+o(d^{-1}))^d-1|+|(1+o(d^{-1}))^d-1|=o(1).\]
We have verified the sufficiency part of the theorem.

Let us prove the necessity part, i.e. \eqref{eq:mpp}$\Rightarrow$\eqref{eq:14}. Let \eqref{eq:mpp} holds and, for each $d,n$, $X_\alpha=X_\alpha(d,n)$ are independent copies of   $X$ with $\e X=0,$ $\e X^2=1>\p(X^2=1)$. So, by Theorem \ref{th:1}, $ \x_p^\top\x_p /p\to 1$ in probability. Taking $Z=X^2$ in Theorem \ref{th:3}, we see that
\begin{center}
$\x_p^\top\x_p/p$ and $S_n^{(d)}/\binom{n}{d}$ from Theorem \ref{th:3}  have the same distribution. 
\end{center}
As convergence in distribution to a constant implies convergence in probability to the same constant, (i) in Theorem \ref{th:3} holds. By (i)$\Rightarrow$(iii) in Theorem \ref{th:3}, \eqref{eq:14} holds.
\end{proof}

\begin{proof}[Proof of Theorem \ref{th:4}.] First, we will show that any of the conditions (i), (ii), (iii) implies that $d^2=o(n)$. Indeed, if (iii) holds, then $\e Z\mathbf{1}(dZ\leqslant n)=1-\e Z\mathbf{1}(dZ>n)=1-o(d^{-1})$ and $ d^2(1-o(1))=d^2(\e Z\mathbf{1}(dZ\leqslant n))^2\leqslant d^2\e Z^2\mathbf{1}(dZ\leqslant n)=o(n),$
hereinafter all limits are with respect to $n\to\infty$. 
To prove that both  (i) and  (ii) imply  $d^2=o(n)$, we introduce one more condition: 

(iv) $\p(d(S_n/n-1)<-\varepsilon)\to 0$ for all $\varepsilon>0$.
\\Obviously, (ii) $\Rightarrow$ (iv). Also, (i) $\Rightarrow$ (iv). Indeed, if (i) holds, then 
\begin{equation} \label{mcl}
 d(S_n/n-1)\geqslant  d\ln\frac{ S_n}{n}\geqslant \ln \frac{S_n^{(d)}}{\binom{n}{d}}=o_\p(1), 
 \end{equation}
where we have used Maclaurin's inequality $S_n/n\geqslant \big(S_n^{(d)}/\binom{n}{d} \big)^{1/d}$  (see (12.3) in \cite{St}). 

Let us prove that (iv) $\Rightarrow$ $d^2=o(n)$. By Theorem 1 in \cite{F} (or Theorem 1.1 in \cite{G}), $\p(S_n-n\leqslant 1)\geqslant 1/13$ for all $n$. By standard inequalities for concentration functions (see Theorem 2.22 in \cite{P}), there exists $C>0$ such that $\p(1-\lambda\leqslant S_n-n\leqslant 1)\leqslant C(\lambda+1)/\sqrt{n}$ for all $\lambda\geqslant 0$ and $n$.  
Taking $\lambda=\gamma\sqrt{n}+1$ for small enough $\gamma>0$, we will guarantee that $\p(- \gamma \sqrt{n}\leqslant S_n-n\leqslant 1)\leqslant1/26$ for all large enough $n$. As a result, we will get that   $\p( S_n-n\leqslant - \gamma \sqrt{n})\geqslant1/26+o(1)$. If (iv) holds, then $ \p( S_n-n\leqslant - \varepsilon  n/d)\to 0$ for any $\varepsilon >0$. The latter is possible only if, for any fixed $\varepsilon>0$, $\varepsilon n/d>\gamma \sqrt{n}$ for all large enough $n$, i.e. $n\geqslant n_0(\varepsilon)$. The latter means that $\sqrt{n}=o(n/d)$ or, equivalently, $d^2=o(n).$ 

Assume further that $d^2=o(n)$. Letting $b_n:=n/d$, we see that $b_n\to\infty$ and $nb_n^{-2}\to0$. By classical weak laws of large numbers (see Theorem 2 in \cite{H} with a discussion above it and page 317 in \cite{Lo}),   $(S_n-n)/b_n=o_\p(1)$  if and only if 
\[n\p( |Z-1|>b_n )+nb_n^{-2} \e (Z-1)^2 \mathbf{1}( |Z-1|\leqslant b_n )+nb_n^{-1}  |\e (Z-1)  \mathbf{1}(|Z-1|\leqslant  b_n )|\to 0 .\]
The last condition could simplified as follows. As $Z\geqslant 0$ a.s. and $b_n\to\infty$, then $Z>1$ when $|Z-1|>b_n$ for any large enough  $n$. This and $\e Z=1$ imply that (for large $n$)  
\[n\p( |Z-1|>b_n )=n\p(  Z-1 >b_n )\leqslant  nb_n^{-1}\e  (Z-1) \mathbf{1}((Z-1)>b_n )=\]\[=
-nb_n^{-1} \e  (Z-1) \mathbf{1}((Z-1)\leqslant b_n ).\]
We also have that $nb_n^{-2} \e (Z-1)^2 \mathbf{1}( |Z-1|\leqslant b_n,Z-1<0)\leqslant nb_n^{-2}=o(1).$
Therefore, the above results yield that (ii) $\Leftrightarrow$ $n\e \phi( (Z-1)/b_n)\to 0$,
where \[\phi(x)=\frac{x^2}2\mathbf{1}(x\in[0,1])+\Big(x-\frac{1}2\Big)\mathbf{1}(x> 1),\quad x\in \bR.\] Using that $\phi $ is non-decreasing and convex, we derive
\[0\leqslant \e \phi(Z/b_n)-\e\phi((Z-1)/b_n)\leqslant 
 \e\phi'(Z/b_n)/b_n=\]\[=b_n^{-2}\e  Z \mathbf{1}(Z\leqslant b_n)+b_n^{-1}\p(Z>b_n)\leqslant 2b_n^{-2}\e  Z =2b_n^{-2}=o(n^{-1}).\]
In addition, $\kappa(x)/2\leqslant\phi(x)\leqslant \kappa(x)$ for $\kappa(x)=x^2\mathbf{1}(x\in[0,1])+x\mathbf{1}(x> 1)$ and all $x\in\bR$.
Thus, (ii) $\Leftrightarrow$ $n\e \phi( (Z-1)/b_n)\to 0$ $\Leftrightarrow$ $n\e \phi(  Z /b_n)\to 0 \Leftrightarrow$ (iii).
   
Let us now show that (i) $\Leftrightarrow$ (ii). By the above arguments, we could assume that $d^2=o(n).$ Also, suppose for a moment that $d\to\infty$. Then, by Lemma \ref{l4}, we have the representation \eqref{eq:17}. Denote further by $\bar \e f(Z)$ the empirical mean $n^{-1}\sum_{k=1}^nf(Z_k)$ for any function $f$. Also, let $\epsilon_n,\varepsilon_n,\delta_n,\gamma_n$ be random sequences tending to 0 in probability, not necessarily the same at each occurrence. As shown in the proof of Lemma \ref{l4}, $\rho d/n\pto 1.$ Using the well-known inequalities  $\ln(x)\leqslant x-1$ and $x \leqslant (x+1)\ln(x+1)\leqslant x(2+x)/2$   valid for every $x\geqslant 0$, we see that by the definition of $\rho$,
\[\frac{d}{n}=  \bar\e  \frac{Z /\rho}{Z /\rho+1}\leqslant  \bar\e \ln\Big(\frac{Z}{ \rho}+1 \Big) \leqslant   \bar\e  \frac{ (Z /\rho)(2+ Z /\rho)}{2(Z /\rho+1)}= \frac{d}{n}+
 \bar\e \frac{ Z^2 /(2\rho)}{  Z+\rho }=\frac{d}{n}+\Big(\frac 1 2+\gamma_n\Big)\bar\e \frac{ dZ^2 }{ Z+\rho  } \]
and
\[ d \ln\frac{ \rho d}n =(1+\varepsilon_n)d \Big(\frac{\rho d}n-1\Big)=(1+\varepsilon_n)\Big(d\bar\e\frac{Z\rho}{Z+\rho}-d\Big)  =(1+\varepsilon_n)\Big( d\Big(\frac{S_n}n-1\Big)- d \,\bar\e  \frac{Z ^2}{Z+\rho }\Big) .\]
As a result, when $d\to\infty$, we get from Lemma \ref{l4} that for $n\geqslant 1,$
\begin{equation}\label{eq:21} (1+\varepsilon_n)\Big(d(S_n/n-1)-d\bar\e  \frac{Z^2}{ Z+\rho }\Big)+\delta_n \leqslant \ln\frac{S_n^{(d)}}{\binom{n}{d}},
\end{equation}
\begin{equation}\label{eq:22}
\ln\frac{S_n^{(d)}}{\binom{n}{d}} \leqslant(1+\varepsilon_n)d(S_n/n-1) -(1/2+\epsilon_n)d\bar\e  \frac{Z^2}{ Z+\rho}+\delta_n.
\end{equation} 

Suppose (ii) holds and $d\to\infty$. Then (iii) holds and, by $\rho d/n\pto 1,$
\[ d\bar\e \frac{Z ^2 }{ Z +\rho  } \leqslant   d\bar\e  Z  I(dZ > n  )+\frac{ 1}{  \rho d/n } \frac{d^2}{n } \bar\e  Z ^2 I(dZ \leqslant n )\pto 0.\]
Therefore, in view of \eqref{eq:21} and \eqref{eq:22}, (i) holds.

Suppose (ii) holds and $d\not\to\infty$. We need the following elementary fact: 
\begin{center} if
$\xi_n\pto 0$, then there exist  (nonrandom) $ \alpha_n\in\bN$, $n\geqslant 1$, such that $\alpha_n\to\infty$ and $\alpha_n\xi_n\pto 0$.
\end{center} Therefore, by (ii), there exists  $d_*=d_*(n)\in\bN$ such that $d(n)\leqslant d_*(n)$ for all $n$, $d=o(d_*)$ (in particular, $d_*\to\infty$), and (ii) holds for $d$ replaced by $d_*$. As shown above, the latter implies that (i) holds for $d$ replaced by $d_*$. By Maclaurin's inequality  (see (12.3) in \cite{St}),
\[\frac{S_n}n-1\geqslant\ln \frac{S_n}n\geqslant\frac1d \ln\frac{S_n^{(d)}}{\binom{n}{d}} \geqslant \frac1{d_*} \ln\frac{S_n^{(d_*)}}{\binom{n}{d_*}} ,\]
where $\ln 0=-\infty$. 
Combining the above relations yields (i) (for $d=d(n)$). 

Suppose (i) holds and $d\to\infty$. Then, noting that  $(z+n/d)/(z+\rho) $ always lies between $1$ and $n/(\rho d)=1+o_\p(1)$ for $z\geqslant 0$, we get from \eqref{eq:22} that  
\begin{equation}\label{eq:15}
d(S_n/n-1)-(1/2+\varepsilon_n) \bar\e \frac{dZ ^2}{Z +n/d }\geqslant \delta_n.
\end{equation} 
It follows from \eqref{eq:15} and $ \bar\e Z ^2/(Z +n/d)\geqslant\bar\e Z  I(dZ > n  )/2+\bar\e Z ^2I(dZ \leqslant n )/(2n/d) $
that
\begin{equation}\label{eq:23} d(S_n/n-1)=d(\bar\e Z-\e Z)\geqslant (1/4+\epsilon_n)\Big(d \bar\e  Z \mathbf{1}(dZ>n)+ \frac{d^2}{n}\,\bar\e  Z^2 \mathbf{1}(dZ\leqslant n)\Big)+ \delta_n.
\end{equation} 
As $\var((d^2/n)\bar\e  Z ^2 \mathbf{1}(dZ \leqslant n))\leqslant d^2(d/n)^2\e Z ^4\mathbf{1}(dZ \leqslant n)/n \leqslant (d^2/n)\e Z ^2\mathbf{1}(dZ \leqslant n)  $, the Chebyshev inequality yields
\[d(S_n/n-1)\geqslant (1/4+\epsilon_n) \frac{d^2}{n}\, \e  Z ^2 \mathbf{1}(dZ\leqslant n)+O_\p(1)\sqrt{\frac{d^2}{n}\, \e  Z ^2 \mathbf{1}(dZ\leqslant n)}+ \delta_n.\] 

Let us show that $d^2\e  Z ^2 \mathbf{1}(dZ\leqslant n)=O(n)$. Suppose the contrary: for some $n_k\to\infty$ (as $k\to\infty$) and $d_k=d(n_k)$, 
$(d_k^2/n_k)\e  Z ^2 \mathbf{1}(d_kZ\leqslant n_k)\to\infty$. If so, the right-hand side of the last inequality (with $n_k,d_k$ replacing $n,d$) goes to infinity in probability. But this is impossible, because by Theorem 1 in \cite{F}, $\p( d_k(S_{n_k}/n_k-1)\leqslant  d_k/n_k)=\p( S_{n_k}\leqslant  n_k+1)\geqslant 1/13$ for all $k$. Thus,  $d^2\e  Z ^2 \mathbf{1}(dZ\leqslant n)=O(n)$.

The variables $d(\bar \e Z\mathbf{1}(dZ\leqslant n)-\e Z\mathbf{1}(dZ\leqslant n))$ have uniformly bounded second moments not exceeding $(d^2/n)\e  Z ^2 \mathbf{1}(dZ\leqslant n)$
and, hence, are bounded in probability. Therefore, \eqref{eq:23} yields that 
\begin{equation}\label{eq:24} (3/4-\epsilon_n)d (\bar \e Z \mathbf{1}(dZ>n )-\e Z\mathbf{1}(d Z>n )) \geqslant  (1/4+\epsilon_n)d \e Z\mathbf{1}(d Z>n )+O_\p(1).
\end{equation}
The latter proves that $d\e  Z  \mathbf{1}(dZ>n)=O(1)$. Indeed, suppose the contrary: for some $n_k\to\infty$ and $d_k=d(n_k)$, 
$d_k\e  Z \mathbf{1}(d_kZ> n_k)\to\infty$. If so,  the right-hand side of \eqref{eq:24}  goes to infinity in probability. But this is impossible, as by Theorem 1 in \cite{F}, $\p( d(\bar \e Z \mathbf{1}(dZ>n )- \e Z\mathbf{1}(d Z>n )))\leqslant (d/n) \e Z\mathbf{1}(d Z>n ) )\geqslant 1/13$ and $ \e Z\mathbf{1}(d Z>n )\leqslant \e Z=1$.  

Let $\xi_n:=d(S_n/n-1)$, $\eta_n:=d(\bar \e Z\mathbf{1}(dZ\leqslant n)- \e Z\mathbf{1}(dZ\leqslant n))$, and $\zeta_n:=d\bar \e Z\mathbf{1}(dZ>n)$.
Obviously, $\e\xi_n=\e \eta_n=0$, $\xi_n=\eta_n+\zeta_n-\e\zeta_n$, and $\zeta_n\geqslant 0$ a.s. As we have argued above, $\p(\xi_n<-\varepsilon)\to 0$  for all $\varepsilon>0$ and $\e \eta_n^2,\e\zeta_n$ are bounded by some constant $K>0.$ To prove (ii), we only need to show that $ \p(\xi_n>\varepsilon)\to 0$  for all $\varepsilon>0$.
The well-known formula for the expectation $\e \xi_n=\int_{\bR_+}\p(\xi_n>x)dx-\int_{\bR_+}\p(\xi_n<-x)dx$ gives
\[\int_0^\infty \p(\xi_n>x)dx=\int_0^\infty \p(\xi_n<-x)dx\leqslant  \varepsilon +(M-\varepsilon)\p(\xi_n<-\varepsilon)+\int_M^\infty \p(\xi_n<-x)dx\]
for all $\varepsilon>0$ and $M\geqslant \varepsilon$. When $x>2K$, we have $\e\zeta_n-x/2<0$ and $\zeta_n-\e\zeta_n>-x/2$ a.s., therefore, for such $x$, $\xi_n<-x$ implies $\eta_n<-x/2$ and $\p(\xi_n<-x)\leqslant\p(\eta_n<-x/2)$. As a result, taking $M=\max\{2K,\varepsilon,4K/\varepsilon\}$ and applying Chebyshev's inequality, we get
\[\int_0^\infty \p(\xi_n>x)dx\leqslant  \varepsilon +(M-\varepsilon)\p(\xi_n<-\varepsilon)+\int_M^\infty \frac{4Kdx}{x^2}\leqslant 2\varepsilon +M\p(\xi_n<-\varepsilon).\]
Therefore, for all $\varepsilon>0$,
\[\varlimsup_{n\to\infty}\int_0^\infty \p(\xi_n>x)dx\leqslant 2\varepsilon.\]
This is possible only if $\int_0^\infty \p(\xi_n>x)dx\to0$. Thus, $\p(\xi_n>\varepsilon)\leqslant \varepsilon^{-1} \int_0^\infty \p(\xi_n>x)dx\to 0$ for all $\varepsilon>0$. We have shown that (i)$\Rightarrow$(ii) when $d\to\infty$. 

Assume that (i) holds and $d\not\to\infty$. To prove (ii), suppose the contrary: (ii) does not hold, i.e. there are $\varepsilon,\delta>0$ and $(n_k)_{k=1}^\infty$ such that $\p(|d_k(S_{n_k}/n_k-1)|>\varepsilon)\geqslant \delta$ for all $k$ (here $d_k=d(n_k)$) and  $n_k\to\infty$ as $k\to\infty$. Such  $ d_k $ are unbounded over $k$ (if not, $d_k(S_{n_k}/n_k-1)\to 0$ a.s. and in probability by the strong law of large numbers). As a result, there is a subsequence $(n_{k_l})_{l=1}^\infty$ such that $d_{k_l}\to\infty$ and $d_{k_l}(S_{n_{k_l}}/n_{k_l}-1)\not\to 0$ in probability  as $l\to\infty$. However, it is shown above that (i) and $d\to\infty$ imply (ii). The same argument shows that the latter holds for the subsequence $(n_{k_l})_{l=1}^\infty$ and $d_{k_l}=d(n_{k_l})$.   We get the contradiction. Thus, (i)$\Rightarrow$(ii). This finishes the proof of the theorem.  \end{proof}

\section*{Supplementary material}
\begin{proof}[Proof of Lemma \ref{l4}.] We will follow the proof of Theorem in  \cite{HS}. 
For any $\rho>0$, we have 
\[2\pi S_n^{(d)}=	 \int_{-\pi}^\pi  \sum_{k=0}^nS_n^{(k)}\frac{(\rho e^{ i\theta})^{n-k}}{(\rho e^{ i\theta})^{n-d}}  d\theta  = \int_{-\pi}^\pi  \prod_{k=1}^n(Z_k+\rho e^{i\theta})\frac{d\theta}{(\rho e^{ i\theta})^{n-d}}= \int_{-\pi}^\pi                                                                                                                                                                                                                                                                                                                                                                                                          e^{u(\theta)}d\theta,\] where $S_n^{(0)}:=1$, $u(\theta)=u(\theta,n,\rho,Z_1,\ldots,Z_n):=n\bar \e  \ln(Z +\rho e^{i\theta})-(n-d)(\ln\rho+i \theta) $, and $\bar \e f(Z)$ denotes  the empirical mean  $n^{-1}\sum_{k=1}^nf(Z_k)$ for any function $f$.
For such $u$,  
\[u'(0)=in\bar \e \frac{\rho}{ Z +\rho }-i(n-d),\quad u''(0)=-n\bar \e\frac{\rho}{ Z+\rho}+n\bar \e\frac{\rho^2}{( Z+\rho )^2}=-n\bar\e   \frac{\rho Z }{( Z +\rho )^2},\]
\[| u'''(\theta)|= \Big|n\bar\e \frac{\rho Z (\rho e^{i\theta}-Z )}{( Z +\rho e^{i\theta} )^3}\Big|\leqslant 
    n\bar\e  \frac{\rho Z}{(\rho \cos\theta+Z )^2}\leqslant     n\bar\e  \frac{\rho Z }{(\rho/\sqrt{2} +Z)^2}\leqslant 2|u''(0)| \] for every $|\theta|\leqslant \pi/4 $, where we have also used that $| \rho e^{i\theta}-Z_k | \leqslant | \rho e^{i\theta}+Z_k| $ for such $\theta$.  
    
Let $\rho=\rho(n,d,Z_1,\ldots,Z_n)>0$ solve $u'(0)=0$, i.e.
\begin{equation}\label{eq:16}
\bar\e \frac{\rho}{ Z +\rho }= 1-\frac{d}n \quad\text{or, equivalently,}\quad \bar\e \frac{Z}{ Z +\rho }= \frac{d}n  ,
\end{equation} 
if such solution exists  and $\rho =1$ otherwise.  By the weak law of large numbers,    
  \[    \inf_{r>0 } \bar\e  \frac{r}{ Z +r}= \bar\e \mathbf{1}(Z =0)\leqslant \p(Z=0)+\varepsilon<1-\frac dn<1=\sup_{r>0 } \bar\e  \frac{r}{ Z+r  }\]
with probability $1-o(1)$ as $n\to\infty$, where $\varepsilon=\p(Z\neq 0)/2$. So, $\p(\rho\text{ satisfies \eqref{eq:16}})\to 1$. We have $\p(\rho>c)\to 1$ for every $c>0,$ since with probability $1-o(1)$,
  \[\bar\e \frac{c}{Z+c}\leqslant \e \frac{c}{Z+c}+\varepsilon <1-\frac{d}{n}=\bar \e \frac{\rho}{Z+\rho}\]
for $\varepsilon=(1-\e c/(Z+c))/2$.  Moreover, $\rho d/n\pto 1$, as  for all $c>0$ with probability $1-o(1)$,
\[\e \frac{Zc }{Z+c}+o_\p(1)=\bar\e \frac{Zc }{Z+c}\leqslant \frac{\rho d}{n}=\bar\e \frac{Z\rho }{Z+\rho}\leqslant \bar\e Z=1+o_\p(1)\]  
and, by the dominated convergence theorem, $\e (Zc)/(Z+c)\to 1$ as $c\to\infty.$

As $\p(Z_k\geqslant 0)=1$ and the function $|Z_k+\rho e^{i\theta}|=(\rho^2+Z^2+2\rho Z\cos\theta)^{1/2}$ is decreasing in $|\theta|\in [0,\pi]$ for all $k$, then $|e^{u(\theta)}|$ is decreasing in $|\theta|$ and   
\[\frac1{2\pi}\int_{\delta<|\theta|\leqslant  \pi}  |e^{u(\theta)}|d\theta\leqslant |e^{u(\delta)}|\quad\text{for any  $\delta>0$}.\] 
Assume further that $u'(0)=0$ (this happens with probability $1-o(1)$, as is shown above). By Taylor's formula with the reminder in the integral form for all $\delta\in[0,\pi/4]$, 
\[\big|e^{u(\delta)-(u(0)+u''(0)\delta^2/2)}\big|=\big|e^{v(\delta)\delta^3 }\big|\leqslant e^{|v( \delta) \delta^3 |} \leqslant e^{ -u''(0)\delta^3 },\]
where $ v(\delta):=\int_0^1 u'''(t\delta)(1-t)^2dt/2$ satisfies $|v(\delta)|\leqslant -u''(0)/2.$
Fix arbitrary  $\delta\in (0,\pi/4]$. We have
$|S_n^{(d)}-(2\pi)^{-1}\int_{-\delta}^{\delta} e^{u(\theta)}d\theta|\leqslant e^{u(0)+u''(0)(\delta^2/2-\delta^3)}$. Also, by Taylor's formula,
\[  \int_{-\delta}^{\delta} e^{u(\theta)}d\theta=e^{u(0)} \int_{-\delta}^{\delta} e^{u''(0)\theta^2/2+v( \theta )\theta^3}   d\theta 
=\frac{e^{u(0)}}{ U} \int_{-\delta U}^{ \delta U} e^{-z^2 (1+r(z))/2}   dz ,\]
where $U:=\sqrt{-u''(0)}$ and $r(z):=-2v(z/U)z/U^{3}$ satisfies $|r(z)| \leqslant \delta$ for  $z\in[-\delta U,\delta U]$. 

Suppose for a moment that $U\pto\infty$. Taking $\delta=U^{-1/2}\wedge(1/2)$, we get that $|r(z)|<1/2$ and $|e^{-z^2 (1+ r(z))/2}|= e^{-z^2/2}|e^{-z^2 r(z) /2}| \leqslant e^{-z^2/2} e^{|z^2 r(z)| /2}\leqslant  e^{-z^2 /4}$ when $|z|\leqslant  \delta U$. Thus,   
\[\Big|\int_{|z|> \delta U}  e^{-z^2 (1+ r(z))/2}   dz\Big|\leqslant
\int_{|z|> \delta U}  e^{-z^2 /4}    dz=o(1)\]
(here  $\delta U\pto \infty$ by the choice of $\delta$). 
In addition, using the inequality $|e^{z}-1|\leqslant|z|e^{|z|}, $ $z\in\bC$, that follows from the series expansion of $e^z$, we conclude that  
\[\Big|\int_{\bR}   (e^{-z^2 (1+ r(z))/2} -e^{-z^2/2})   dz\Big|\leqslant 
 \int_{\bR}z^2|r(z)| e^{z^2(-1+| r(z)|)/2}  dz \leqslant \delta  \int_{\bR} z^2 e^{-z^2/4}dz\pto0  \]
because $\delta=U^{-1/2}\wedge(1/2)\pto 0$.   
 Combining the above bounds along with $e^{-U^2(\delta^2/2-\delta^3)}=o_\p(1)/U$, we see that if $-u''(0)=U^2\pto \infty$, then 
 \[S_n^{(d)}=\frac{e^{u(0)}}{2\pi \sqrt{- u''(0) }}\int_\bR e^{-z^2}dz (1+o_\p(1))=\frac{e^{u(0)}}{ \sqrt{-2\pi u''(0) }}(1+o_\p(1)).\]
In addition, when $d^2=o(n)$, we have $\binom{n}{d}=(1+o(1)) n^d/d! =(1+o(1))( 2\pi d )^{-1/2} (en)^d/ d^d$ by the Stirling formula.
Thus, assuming $-u''(0)\pto \infty, $ then we get the asymptotic formula
\[\ln\frac{S_n^{(d)}}{\binom{n}{d}}-(u(0)-(\ln(-u''(0)/d))/2-d-d\ln(n/d))\pto 0.\]

To finish the proof, we will check that $-u''(0)/d\pto 1$ and, in  particular,  $-u''(0)\pto \infty $.
Using Chebyshev's sum inequality $n\sum_{k=1}^n a_kb_k\leqslant  \sum_{k=1}^na_k\sum_{k=1}^nb_k$ that holds for every $a_1\leqslant \ldots\leqslant a_n$ and  $b_1\geqslant \ldots\geqslant b_n$, we conclude that with probability $1-o(1)$,
   \[-u''(0)= n\bar\e \frac{ Z}{  Z+\rho  }\frac{\rho}{   Z+\rho   }\leqslant n \bar\e \frac{Z}{  Z+\rho  }\bar\e\frac{\rho }{   Z+\rho   }=n(d/n)(1-d/n)=d(1+o(1)).\] In addition, by the Cauchy-Schwartz inequality, with probability $1-o(1)$,
 \[-u''(0)=  n\rho \bar\e \frac{ Z }{ ( Z+\rho)^2  }\geqslant 
\frac{n \rho}{ \bar\e Z } \Big( \bar\e  \frac{ Z }{   Z +\rho  }\Big)^2=\frac{n( \rho d/n)d/n}{\e Z+o_\p(1)}=d(1+o_\p(1))\to\infty\]
(recall that $\e Z=1$ and $\rho d/n\pto 1$). 
This finishes  the proof of the lemma.
\end{proof}

\begin{proof}[Proof of Lemma \ref{l3}]
Setting $[n]_d:=\{i\subseteq\{1,\ldots,n\}:|i|=d\}$, fix arbitrary $i_0,j_0\subseteq [n]_d$ with $|i_0\cap j_0|=t$. By definition,  $\gamma(s,t)$ is the number of pairs $(k_0,l_0)$ such that $k_0,l_0\subseteq [n]_d$,  $|k_0\cap l_0|=t$, $|\Lambda_3(i_0,j_0,k_0,l_0)|/2+|\Lambda_4(i_0,j_0,k_0,l_0)|=s,$ and $|\Lambda_1(i_0,j_0,k_0,l_0)|=0$.  Let us count such pairs. 

Set further $(i,j,k,l,u,v)=(i_0\setminus j_0,j_0\setminus i_0,k_0 \setminus l_0,l_0 \setminus k_0,i_0\cap j_0,k_0\cap l_0)$. By definition, 
\begin{equation}\label{def}
\text{$i,j,u$ are pairwise disjoint, $|i|=|j|=d-t$, and $|u|=t$ (the same for $k,l,v$).}
\end{equation}
 Recall  that $\Lambda_c(i_0,j_0,k_0,l_0),$ $c=1,\ldots,4,$ contains all $\alpha\in i_0\cup j_0\cup k_0\cup l_0$ that are covered by exactly $c$ sets among $i_0,j_0,k_0,l_0$. Therefore, 
\[|\Lambda_1|=0\quad\text{iff}\quad i\cup j\subseteq k\cup l\cup v  \quad\text{and}\quad k\cup l\subseteq i\cup j\cup u, \]
\[|\Lambda_3|= |((i\cup j)\cap v) \cup (u \cap (k\cup l))|=|(i\cup j)\cap v|+|u \cap (k\cup l)|,\quad |\Lambda_4|=|u\cap v|.\]
Put $q:=|(i\cup j)\cap v|$ and $r:=|u\cap v|$. Let us show that it follows from $|\Lambda_1|=0$ and $|i\cup j|= |k\cup l|$ that $|(i\cup j)\cap v|=|u \cap (k\cup l)|=q$. The relation $i\cup j\subseteq k\cup l\cup v$ implies that
\[(i\cup j)\setminus v=(i\cup j)\setminus (u\cup v) \subseteq (k\cup l\cup v)\setminus (u\cup v)=(k\cup l)\setminus u.\] Similarly, 
$k\cup l\subseteq i\cup j\cup u$ implies that $(k\cup l)\setminus u\subseteq (i\cup j)\setminus v$. This proves that
\[q=|(i\cup j)\cap v|=|i\cup j|-|(i\cup j)\setminus  v|=|k\cup l|-|(k\cup l)\setminus u|=|(k\cup l)\cap u|.\]
Combining the above relations gives $|\Lambda_3|=2q$ and $s=q+r=|(i\cup j)\cap v|+|u\cap v|=|(i\cup j\cup u)\cap v|\leqslant |v|=t.$ In particular, this proves that $\gamma(s,t)=0$ when $s>t$. 

Suppose $s\in\{0,\ldots,t\}$. For given $(i,j,u)$ satisfying \eqref{def} and $r\leqslant s$, let us compute the number of triples $(k,l,v)$  satisfying \eqref{def} and such that $|\Lambda_1|=0,$ $|\Lambda_3|=2(s-r)$,  and $|\Lambda_4|=|u\cap v|=r$, where $\Lambda_c=\Lambda_c(i\cup u,j\cup u,k\cup v,l\cup v)$ for $c=1,3,4.$ 

The number of possible choices of $v\cap u$ is $\binom{t}{r},$ as $|u|=t$ and $|u\cap v|=r$.

Given $v\cap u,$ the number of possible choices of $v\cap (i\cup j)$ is $\binom{2(d-t)}{s-r} ,$ as $(i\cup j)\cap u=\varnothing$, $|i\cup j|=2(d-t),$ and $|v\cap (i\cup j)|=s-r$.  

Given $v\cap (i\cup j\cup u),$ the number of possible choices of $v\setminus (i\cup j\cup u)$ is $\binom{n-(2d-t)}{t-s} ,$ as $|i\cup j\cup u|=2d-t$ and $|v\setminus (i\cup j\cup u)|=t-s$.
  
Given $v,$ the number of possible  choices of $(k\cup l)\cap u$ is $\binom{t-r}{s-r} ,$ as  $(k\cup l)\cap u=(k\cup l)\cap (u\setminus v)$, $|u\setminus v|=t-r$, and $|(k\cup l)\cap u|= s-r $.

Given $v$ and $(k\cup l)\cap u$,  the number of possible choices of $(k,l)$ is $\binom{2(d-t)}{d-t},$ as $k\cap l=\varnothing$ and, because of $|\Lambda_1|=0$, $k$ and $l$ could be only composed from the elements of the set $((i\cup j)\setminus v)\cup ((k\cup l)\cap u)$, which has $2(d-t)$ elements in total.

Combining the above bounds and varying $r$, we deduce that
\[\gamma(s,t)=\sum_{r=0}^s \binom{t}{r} \binom{2(d-t)}{s-r} \binom{t-r}{s-r}\binom{n-2d+t}{t-s}\binom{2(d-t)}{d-t} \mathbf{1}(s\leqslant t).\]
As is well known, $\binom{2(d-t)}{m} \leqslant 2^{2(d-t)}$ for any  $m.$
Also, it follows from $s\leqslant t<d\leqslant n$ that
 \[\frac{\binom{n-2d+t}{t-s}}{ \binom{n}{d}}=\frac{\frac{(n-2d+t)!}{(n-2d+s)! }}{\frac{(n-d+(t-s))!}{(n-d)!}} \frac{\frac{d!}{(t-s)!}}{\frac{n!}{(n-d+(t-s))!}}=\prod_{z=1}^{t-s}\frac{ n-2d+s+z}{n-d+z}
 \prod_{m=0}^{d-(t-s)-1}\frac{d-m}{n-m}\leqslant\]\[\leqslant \prod_{m=0}^{d-(t-s)-1}\frac{d-m}{n-m}=\Big(\frac d n\Big)^{d-(t-s)} \prod_{m=0}^{d-(t-s)-1}\frac{1-m/d}{ 1-m/ n  }\leqslant  \Big(\frac d n\Big)^{d-(t-s)},\]
 where $\prod_{z=1}^{t-s}$ is equal to one if $t=s$. As a result, denoting $2^{4(d-t)}\binom{n}{d}( d/n)^{d-(t-s)}\mathbf{1}(s\leqslant t)$ by  $M$, we get the desired bound
\[\gamma(s,t)\leqslant M\sum_{r=0}^s \binom{t}{r}  \binom{t-r}{s-r}  =  M\binom{t}{s}\sum_{r=0}^s \binom{s}{r} =  2^s\binom{t}{s}M .\]
\end{proof}


\end{document}